\documentclass[11pt,reqno]{amsart}
\usepackage{color}

\newtheorem{theorem}{Theorem}[section]
\newtheorem{theoremsub}{Theorem}[subsection]

\newtheorem{proposition}{Proposition}[section]
\newtheorem{corollary}{Corollary}[section]
\newtheorem{corollarysub}{Corollary}[subsection]
\newtheorem{remark}{Remark}[section]
\newtheorem{remarksub}{Remark}[subsection]

\newfont{\bb}{msbm10 at 12pt}

\def\Xep{X_{\varepsilon}}

\def\<{\langle}

\def\>{\rangle}

\newcommand{\bal}{\begin{align}}      \newcommand{\eal}{\end{align}}
\newcommand{\ba}{\begin{array}}      \newcommand{\ea}{\end{array}}
\newcommand{\bc}{\begin{center}}     \newcommand{\ec}{\end{center}}
\newcommand{\be}{\begin{enumerate}}  \newcommand{\ee}{\end{enumerate}}
\newcommand{\beQ}{\begin{eqnarray*}} \newcommand{\eeQ}{\end{eqnarray*}}
\newcommand{\bi}{\begin{itemize}}    \newcommand{\ei}{\end{itemize}}
\newcommand{\bt}{\begin{tabular}}    \newcommand{\et}{\end{tabular}}
\newcommand{\bdm}{\begin{displaymath}} \newcommand{\edm}{\end{displaymath}}

\newcommand{\D}{D\!\!\!\!/\,}

\newcommand{\RSB}{S\!\!\!\!/\,}
\newcommand{\Ss}{R\!\!\!\!/\,}
\newcommand{\ovg}{\overline{g}}
\newcommand{\ovgf}{\overline{\mathfrak{g}}}
\newcommand{\hatgf}{\widehat{\mathfrak{g}}}
\newcommand{\hatg}{\widehat{g}}

\newcommand{\LSs}{\Delta\!\!\!\!/\,}

\newcommand{\nbs}{\nabla\!\!\!\!/\,}
\newcommand{\mult}{\gamma\!\!\!/}

\begin{document}
\title[Scalar flat compactifications of Poincar\'e-Einstein manifolds]{Scalar flat compactifications of Poincar\'e-Einstein manifolds and applications}

\author{Simon Raulot}
\address{Laboratoire de Math\'ematiques R. Salem
UMR $6085$ CNRS-Universit\'e de Rouen
Avenue de l'Universit\'e, BP.$12$
Technop\^ole du Madrillet
$76801$ Saint-\'Etienne-du-Rouvray, France.}
\email{simon.raulot@univ-rouen.fr}

\begin{abstract}
We derive an integral inequality between the mean curvature and the scalar curvature of the boundary of any scalar flat conformal compactifications of Poincar\'e-Einstein manifolds. As a first consequence, we obtain a sharp lower bound for the first eigenvalue of the conformal half-Laplacian of the boundary of such manifolds. Secondly, a new upper bound for the renormalized volume is given in the four dimensional setting. Finally, some estimates on the first eigenvalues of Dirac operators are also deduced. 
\end{abstract}

\date{\today}
\maketitle 
\pagenumbering{arabic}


\section{Introduction}


In the seminal work \cite{GrahamZworski}, Graham and Zworski use scattering theory to study a family of conformally covariant operators defined on the boundary at infinity $(\partial X, [\ovgf])$ of a $(n+1)$-dimensional Poincar\'e-Einstein manifold $(X,g_+)$. These operators denoted by $P^{\alpha}_{\hatgf}$ for $\hatgf\in[\ovgf]$ are usually called the fractional conformal Laplacians and define pseudo-differential self-adjoint operators of order $2\alpha$ for suitable $\alpha\in\mathbb{C}$. The simplest example of a Poincar\'e-Einstein manifold is the hyperbolic space $(\mathbb{H}^{n+1},g_{\mathbb{H}})$ whose boundary at infinity is the M\"obius sphere $(\mathbb{S}^n,[\mathfrak{g}_{\mathbb{S}}])$. Let us explain the construction of these operators when we consider the upper half-space model of the hyperbolic space
\begin{eqnarray*}
\Big(\mathbb{R}^{n+1}_+:=\big\{(x,y)\in\mathbb{R}^n\times\mathbb{R}\,/\,y>0\big\},g_\mathbb{H}=y^2g_{\mathbb{R}}\Big)
\end{eqnarray*} 
where $g_{\mathbb{R}}$ is the Euclidean metric. To avoid technical difficulties, we assume that $0<\alpha<1$ and we refer to Section \ref{HalfConfLaplacian} for the general case. For a function $f:\mathbb{R}^n\rightarrow\mathbb{R}$, the eigenvalue problem
\begin{eqnarray}\label{EigenvalueLaplacian}
-\Delta_{g_\mathbb{H}} u-s(n-s)u = 0 \quad\text{in}\quad\mathbb{H}^{n+1},
\end{eqnarray}
where $\Delta_{g_\mathbb{H}}$ is the Laplacian on the hyperbolic space and $s=\frac{n}{2}+\alpha$, admits a solution of the form
\begin{eqnarray*}
u=Fy^{n-s}+Gy^s, \quad F,G\in C^{\infty}(X),\quad F_{|y=0}=f.
\end{eqnarray*}
If $\mathfrak{g}_{\mathbb{R}}$ denotes the restriction to $\partial\mathbb{R}^{n+1}_+$ of $g_{\mathbb{R}}$, the maps
\begin{eqnarray}\label{FractionalLaplacianGZ}
P^{\alpha}_{\mathfrak{g}_{\mathbb{R}}}f:=d_\alpha S\Big(\frac{n}{2}+\alpha\Big)f,\quad d_\alpha:=2^{2\alpha}\frac{\Gamma(\alpha)}{\Gamma(-\alpha)}
\end{eqnarray}
where $S(s)f=G_{|y=0}$ is the scattering operator, define a family of conformally covariant self-adjoint pseudo-differential operators of order $2\alpha$ on $(\mathbb{R}^n,\mathfrak{g}_{\mathbb{R}})$.

A typical example of such an operator on $\mathbb{R}^n$ is the fractional Laplacian. This is a non local operator of order $2\alpha$ which can be defined by the singular integral (suitably regularized)
\begin{eqnarray*}
(-\Delta_x)^\alpha f(x)=C_{n,\alpha}\int_{\mathbb{R}^n}\frac{f(x)-f(\xi)}{|x-\xi|^{n+2\alpha}}d\xi
\end{eqnarray*}
where $C_{n,\alpha}$ is some normalization constant. In \cite{CaffarelliSilvestre}, Caffarelli and Silvestre proved that this operator can be defined in an alternative way. More precisely, for a function $f:\mathbb{R}^n\rightarrow\mathbb{R}$, they consider the extension problem 
\begin{equation}\label{ExtensionRn}
\left\lbrace
\begin{array}{rll}
\Delta_x U+\frac{a}{y}\partial_yU+\partial_{yy}U & = 0,  & \text{for }(x,y)\in\mathbb{R}^{n+1}_+,\\
U(x,0) & = f(x), & \text{for }x\in\mathbb{R}^n
\end{array}
\right.
\end{equation}
where $U:\mathbb{R}^{n+1}_+\rightarrow\mathbb{R}$ and $a=1-2\alpha$ and they prove the remarkable result that 
\begin{eqnarray*}
(-\Delta_x)^\alpha f=\frac{d_\alpha}{2\alpha}\lim_{y\rightarrow 0}y^a\partial_yU
\end{eqnarray*}
where the constant $d_\alpha$ is defined in (\ref{FractionalLaplacianGZ}). This operator can then be thought as a generalized Dirichlet-to-Neumann operator since for $\alpha=1/2$ this is precisely the well-known Dirichlet-to-Neumann operator. 

Noticing that a function $u$ is a solution to the eigenvalue problem (\ref{EigenvalueLaplacian}) if and only if the function $U:=y^{s-n}u$ is a solution of the extension problem (\ref{ExtensionRn}), Chang and Gonzalez \cite{ChangGonzalez} make the fundamental observation that 
\begin{eqnarray*}
P^{\alpha}_{\mathfrak{g}_{\mathbb{R}}}f=(-\Delta_x)^\alpha f=\frac{d_\alpha}{2\alpha}\lim_{y\rightarrow 0}y^a\partial_yU.
\end{eqnarray*}
This is the key point which allows to generalize this construction in the context of Poincar\'e-Einstein manifolds (see Section \ref{IntegralInequality} for precise definitions). In fact, fractional conformal Laplacians on the conformal infinity $(\partial X,[\ovgf])$ of Poincar\'e-Einstein manifolds $(X,g_+)$ can be defined as generalized Dirichlet-to-Neumann operators on compactifications of $(X,g_+)$.  
 
On the other hand, in the context of Poincar\'e-Einstein manifolds, it is natural to search for relations between the conformal structure of the boundary at infinity and the Riemannian geometry of $(X,g_+)$. Recently with O. Hijazi and S. Montiel \cite{HijaziMontielRaulot7}, we proved that the non negativity of the Yamabe invariant $\mathcal{Y}(\partial X,[\ovgf])$ (see Section \ref{HalfConfLaplacian}) of the boundary at infinity is equivalent to the fact that the linear isoperimetric 
\begin{eqnarray*}
n{\rm Vol}(\Omega)\leq {\rm Area}(\partial\Omega)
\end{eqnarray*}
holds for all compact domains $\Omega$ contained in $X$ where ${\rm Vol}(\Omega)$ and ${\rm Area}(\partial\Omega)$ denote respectively the volume of $\Omega$ and the area of $\partial\Omega$ with respect to $g_+$. An other characterization was previously obtained by Guillarmou and Qing \cite{GuillarmouQing} and state that $\mathcal{Y}(\partial X,[\ovgf])>0$ if and only if the largest real scattering pole of $X$ is less than $\frac{n}{2}-1$. In the same work, they also show that the positivity of the Yamabe invariant implies the positivity of the first eigenvalue of the operator $P^\alpha_{\hatgf}$ for all $0<\alpha\leq 1$ and $\hatgf\in[\ovgf]$. 

In this paper, we prove a sharp inequality between the first eigenvalue of the half-conformal Laplacian $P^{1/2}_{\hatgf}$ and the first eigenvalue of the conformal Laplacian $P^{1}_{\hatgf}$ for any $\hatgf\in[\ovgf]$. This result particularly fits well with the last quoted result of Guillarmou and Qing since the positivity of the first eigenvalue of $P^{1}_{\hatgf}$ is equivalent to the positivity of the Yamabe invariant of the boundary at infinity. Moreover this estimate is sharp since equality for a metric in the conformal infinity occurs if and only if the Poincar\'e-Einstein manifold is the hyperbolic space. The proof of this result relies on a new integral inequality (see Theorem \ref{MainRes}) involving the mean curvature and the scalar curvature of the boundary of any scalar flat compactifications of the Poincar\'e-Einstein manifold. The fact that scalar flat metrics appears naturally in this proof comes from the mere observation that the generalized Dirichlet-to-Neumann operator associated to the half-conformal Laplacian via the previous construction has a nice geometric description. Indeed it precisely deals with the existence, on a compact Riemannian manifold with boundary which in our situation are the compactifications of $(X,g_+)$, of a conformally related metric with zero scalar curvature and constant mean curvature (see Section \ref{HalfConfLaplacian}). As we will see this integral inequality has several other direct applications. For example, it implies that the only Poincar\'e-Einstein manifold which admits a Ricci-flat compactification has to be the hyperbolic space. This generalizes a recent result of Chen, Wang and Zhang \cite{ChenWangZhang} where a similar result is obtained under the additional assumption that the mean curvature is constant. Moreover, in Section \ref{UpperBoundsRenorVol}, we prove a new upper bound for the renormalized volume of four dimensional Poincar\'e-Einstein manifolds and the proof also relies on our main integral inequality. Finally in the last section, we consider spin Poincar\'e-Einstein manifolds and study more precisely the eigenvalue of the boundary Dirac operator on the conformal infinity. In particular, when combined with a conformal lower bound of Hijazi, Montiel and Zhang \cite{HijaziMontielZhang2} and the Hijazi inequality \cite{Hijazi1}, we obtain a new proof of an uniqueness result of Hijazi and Montiel \cite{HijaziMontiel3} concerning Poincar\'e-Einstein manifolds whose conformal infinity carries a twistor spinor (see Theorem \ref{HM-Twistor}). Although our assumptions here are in fact stronger than \cite{HijaziMontiel3} it seems interesting to the author to give an alternative proof of this result. Finally, our previous estimates on the renormalized volume are used to bound from below the first non negative eigenvalue of the boundary Dirac operator of four dimensional spin Poincar\'e-Einstein manifolds (see Theorem \ref{ReVol-BDO1}).


\section{An integral formula}\label{IntegralInequality}


Let $X$ be the interior of a $(n+1)$-dimensional connected and compact manifold $\overline{X}$ with (possibly disconnected) boundary $\partial X$. Recall that $(X,g_+)$ is a conformally compact manifold of order $C^{m,\alpha}$ with $m\geq 3$ and $\alpha\in]0,1[$ if $g_+$ is a Riemannian metric on $X$ such that the conformally related metric $\ovg=\rho^2g_+$ extends to a $C^{m,\alpha}$ Riemannian metric on $\overline{X}$ for a defining function $\rho$ of the boundary $\partial X$. A defining function is a smooth function $\rho:\overline{X}\rightarrow\mathbb{R}$ which is positive on $X$,  zero on $\partial X$ and such that $d\rho_{|\partial X}\neq 0$. There exists actually plenty of defining functions so that a canonical choice of the metric on the boundary is not possible. However, it is clear that such compactifications of $g_+$ induce a well-defined {\it conformal structure} on the boundary $\partial X$ denoted by $[\ovgf]$ where $\ovgf:=\ovg_{|\partial X}$. The conformal manifold $(\partial X,[\ovgf])$ will be refer to the {\it conformal infinity} or to the {\it boundary at infinity} of the conformally compact manifold $(X,g_+)$. Moreover, it is a well-known fact (see \cite{GrahamLee,Lee}) for example) that a metric $\hatgf\in[\ovgf]$ on $\partial X$ determines a unique defining function $r$ in a neighborhood of $\partial X$ such that the metric $\widetilde{g}=r^2g_+$ satisfies $\widetilde{g}_{|\partial X}=\hatgf$ and $|\widetilde{\nabla} r|^2=1$. Here $\widetilde{\nabla}r$ denotes the gradient of the function $r$ with respect to the metric $\widetilde{g}$ whose norm is also taken with respect to $\widetilde{g}$.  Such a function is called the {\it geodesic defining function} associated to $\hatgf$ since the integral curves of the vector field $\widetilde{\nabla} r$ are geodesics orthogonal to the boundary $\partial X$ so that the metric $g_+$ takes the form
\begin{eqnarray}\label{NormalForm}
g_+=r^{-2}\big(dr^2+\widetilde{g}_r\big)
\end{eqnarray} 
in a neighborhood of $\partial X$. Here $\widetilde{g}_r$ is a one-parameter family of metrics on $\partial X$ with $\widetilde{g}_0=\hatgf$. 

If now we assume that $g_+$ is at least $C^2$ and satisfies the Einstein equation
\begin{eqnarray*}
Ric_{g_+}=-ng_+
\end{eqnarray*}
we say that $(X,g_+)$ is a Poincar\'e-Einstein manifold. Note that in this situation if the boundary at infinity admits a smooth representative then it follows from \cite{ChruscielDelayLeeSkinner} that for any smooth defining function, the metric $\overline{g}=\rho^2g_+$ is at least of class $C^{3,\alpha}$ on $\overline{X}$. We will always assume that these assumptions are fulfilled unless otherwise stated. 

In the following, we will denote by $\ovg$, $\hatg$ or $\widetilde{g}$ metrics on the compact manifold with boundary $\overline{X}=X\cup\partial X$ and by $\ovgf$, $\hatgf$ and $\widetilde{\mathfrak{g}}$ their corresponding restrictions to the boundary $\partial X$. 

Now we prove our main inequality (see Theorem \ref{MainRes}) which will be of constant use in this paper. For this, we generalize the arguments used in \cite{ChenLaiWang,GurskyHan,Raulot10} largely inspired by Obata \cite{Obata} in his work on the uniqueness of Yamabe metrics for Einstein closed manifolds. First, we remark that if $\hatg=\rho^2 g_+$ for $\rho$ a smooth defining function, the formula which gives the Ricci curvatures of two conformally related metrics (see \cite{Besse} for example) leads to
\begin{eqnarray}\label{RicConf}
Ric_{\hatg}=-(n-1)\rho^{-1}\widehat{\nabla}^2\rho+\Big(n\rho^{-2}\big(|\widehat{\nabla}\rho|^2-1\big)-\rho^{-1}\widehat{\Delta}\rho\Big)\widehat{g}.
\end{eqnarray}
Here $\widehat{\nabla}$, $\widehat{\nabla}^2$ and $\widehat{\Delta}$ denote respectively the gradient, the hessian and the Laplace operator on $(\overline{X},\hatg)$. Taking the trace of the previous identity with respect to the metric $\hatg$ implies that the scalar curvature $R_{\hatg}$ of $(\overline{X},\hatg)$ is given by
\begin{eqnarray}\label{ScalarConf}
R_{\hatg}=-2n\rho^{-1}\widehat{\Delta}\rho+n(n+1)(|\widehat{\nabla}\rho|^2-1).
\end{eqnarray}
From (\ref{RicConf}) and (\ref{ScalarConf}), we can compute that if the defining function $\rho$ is a geodesic one, the boundary $\partial X$ of $\overline{X}$ is totally geodesic with respect to $\hatg$. This is due to the fact that the first order term in the asymptotic expansion of $\widetilde{g}_r$ in (\ref{NormalForm}), which is precisely twice the second fundamental form of $\partial X$ in $(\overline{X},\widetilde{g})$, is zero by \cite{FeffermanGraham1}. In particular, for any metric $\hatg\in[\ovg]$, $\partial X$ is a totally umbilical hypersurface in $\overline{X}$ with respect to $\hatg$ since this property only depends on the conformal class of $\ovg$. As we will see, this is of great importance in the following. Now denote by
\begin{eqnarray*}
E_{\hatg}:=Ric_{\hatg}-\frac{R_{\hatg}}{n+1}\hatg
\end{eqnarray*}
the trace free part of the Ricci tensor which from (\ref{RicConf}) and (\ref{ScalarConf}) can be written as
\begin{eqnarray}\label{TF-Ricci1}
E_{\hatg}=-(n-1)\Big(\widehat{\nabla}^2\rho-\frac{\widehat{\Delta}\rho}{n+1}\hatg\Big).
\end{eqnarray}
Then for $\varepsilon>0$ sufficiently small and $x\in\partial X\times[0,\varepsilon[\subset\overline{X}$, we let $y(x)={\rm dist}_{\hatg}(x,\partial X)$ in such a way that the metric takes the form 
\begin{eqnarray}\label{GeodCoord}
\hatg=dy^2+\hatg_y
\end{eqnarray}
where $\hatg_y$ is a one-parameter family of metrics on $\partial X$. Moreover, in this coordinate system, the defining function $\rho$ has the following expansion 
\begin{eqnarray}\label{ExpansionDF}
\rho=y+O(y^2).
\end{eqnarray}
Now we consider
\begin{eqnarray*}
X_\varepsilon:=\Big\{x\in X\,/\,y(x)\geq\varepsilon\Big\},
\end{eqnarray*}
and we compute using (\ref{TF-Ricci1}) that
\begin{eqnarray*}
\int_{\Xep}\rho|E_{\hatg}|^2dv_{\hatg} & = & -(n-1)\int_{\Xep}\<\Big(\widehat{\nabla}^2\rho-\frac{\widehat{\Delta}\rho}{n+1}\hatg\Big),E_{\hatg}\>dv_{\hatg}\\
& = & -(n-1)\int_{\Xep}\<\widehat{\nabla}^2\rho,E_{\hatg}\>dv_{\hatg}
\end{eqnarray*}
since $E_{\hatg}$ is trace-free. An integration by parts gives
\begin{eqnarray*}
\int_{\Xep}\rho|E_{\hatg}|^2\,dv_{\hatg}=-(n-1)\int_{\Xep}\<\widehat{\nabla}\rho,\widehat{\delta} E_{\hatg}\>\,dv_{\hatg}-(n-1)\int_{\partial\Xep}E_{\hatg}\big(\widehat{\nabla}\rho,\widehat{N}_\varepsilon\big)\,ds_{\hatgf}
\end{eqnarray*}
where $\widehat{\delta}$ denotes the divergence of a tensor field, $dv_{\hatg}$ (resp. $ds_{\hatgf}$) is the Riemannian volume element of $(\overline{X},\hatg)$ (resp. of $(\partial X,\hatgf)$) and $\widehat{N}_\varepsilon$ is the outward unit normal to $\partial\Xep$. From the second Bianchi identity, we deduce that  
\begin{eqnarray*}
\widehat{\delta} E_{\hatg}=-\frac{n-1}{2(n+1)}\widehat{\nabla} R_{\hatg}
\end{eqnarray*}
and this implies by letting $\varepsilon\rightarrow 0$ the following general formula
\begin{eqnarray*}
\int_{\overline{X}}\rho|E_{\hatg}|^2\,dv_{\hatg}=\frac{(n-1)^2}{2(n+1)}\int_{\overline{X}}\<\widehat{\nabla}\rho,\widehat{\nabla} R_{\hatg}\>\,dv_{\hatg}+(n-1)\int_{\partial X}E_{\hatg}\big(\widehat{N},\widehat{N}\big)\,ds_{\hatgf}.
\end{eqnarray*}
Here we used the facts that $\widehat{\nabla}\rho=\frac{\partial}{\partial y}+O(y)$ and that $\widehat{N}:=-\frac{\partial}{\partial y}$ is the outward unit normal along $\partial X$ in $(\overline{X},\hatg)$ because of (\ref{GeodCoord}). If now we assume that $R_{\hatg}$ is constant, the previous equality reads as 
\begin{eqnarray}\label{GeneralId}
\int_{\overline{X}}\rho|E_{\hatg}|^2\,dv_{\hatg}=(n-1)\int_{\partial X}E_{\hatg}\big(\widehat{N},\widehat{N}\big)\,ds_{\hatgf}. 
\end{eqnarray}
On the other hand, if $\Ss_{\hatgf}$ denotes the scalar curvature of $(\partial X,\hatgf)$, the Gauss formula gives
\begin{eqnarray}\label{GaussFormula}
\Ss_{\hatgf}=R_{\hatg}-2 Ric_{\hatg}(\widehat{N},\widehat{N})+n(n-1)H^2_{\hatg}
\end{eqnarray}
since $\partial X$ is totally umbilical so that the integrand in the right-hand side of (\ref{GeneralId}) can be expressed as
\begin{eqnarray}\label{TraceFreeRicci}
E_{\hatg}(\widehat{N},\widehat{N})=\frac{1}{2}\Big(\frac{n-1}{n+1}R_{\hatg}-\Ss_{\hatgf}+n(n-1)H^2_{\hatg}\Big).
\end{eqnarray}
Here $H_{\hatg}=\frac{1}{n}{\rm Tr}_{\hatg}(L_{\hatg})$ is the (inner) mean curvature of $\partial X$ in $(\overline{X},\hatg)$ defined by taking the normalized trace of the second fundamental form $L_{\hatg}$. Combining (\ref{GeneralId}) and (\ref{TraceFreeRicci}) implies that 
\begin{eqnarray}\label{MainCSC}
\frac{n-1}{n+1}R_{\hatg}\,{\rm Vol}(\partial X,\hatgf)+\int_{\partial X}\Big(n(n-1)H^2_{\hatg}-\Ss_{\hatgf}\Big)\,ds_{\hatgf}\geq 0
\end{eqnarray}
for all compactified metrics $\hatg$ with constant scalar curvature. Moreover equality occurs if and only if the metric $\hatg$ is Einstein. From this observation, we deduce the main result of this section:
\begin{theorem}\label{MainRes}
Let $(X,g_+)$ be a Poincar\'e-Einstein manifold which has a smooth representative in its conformal infinity $(\partial X,[\ovgf])$. Then for any  defining function $\rho$ of $\partial X$ such that the scalar curvature $R_{\hatg}$ of the compactified manifold $(\overline{X},\hatg)$ is zero where $\hatg=\rho^2g_+$, we have
\begin{eqnarray}\label{MainIneBis}
n(n-1)\int_{\partial X} H^2_{\hatg}\,ds_{\hatgf}\geq\int_{\partial X}\Ss_{\hatgf}\,ds_{\hatgf}.
\end{eqnarray}
Moreover, there exists a compactification of $(X,g_+)$ for which (\ref{MainIneBis}) is an equality if and only if $(X,g_+)$ is isometric to the hyperbolic space $(\mathbb{H}^{n+1},g_{\mathbb{H}})$. 
\end{theorem} 

\begin{proof}
The integral inequality (\ref{MainIneBis}) follows directly using (\ref{MainCSC}) for any scalar flat compactification $(\overline{X},\hatg)$ of the Poincar\'e-Einstein space $(X,g_+)$. If equality occurs, it follows from (\ref{GeneralId}) and (\ref{TraceFreeRicci}) that $E_{\hatg}$ vanishes on $\overline{X}$ and since $\hatg$ is scalar flat it has to be Ricci-flat. In particular, we have 
\begin{eqnarray}\label{EqMS}
n(n-1)H_{\hatg}^2=\Ss_{\hatgf}.
\end{eqnarray}
From (\ref{TF-Ricci1}), we also deduce that $\rho$ satisfies the boundary value problem 
\begin{equation}\label{EqualityDF}
\left\lbrace
\begin{array}{ll}
\widehat{\nabla}^2\rho=\frac{\widehat{\Delta}\rho}{n+1}\hatg & \text{ on  } \overline{X}\\
\rho_{|\partial X}=0,\,\,\widehat{N}(\rho)_{|\partial X}=-1 & \text{ on  }\partial X
\end{array}
\right.
\end{equation}
where the boundary conditions follow from (\ref{ExpansionDF}) and the fact that on $\partial X$, it holds that $\widehat{N}=-\frac{\partial}{\partial y}$. Now from the Ricci identity and the first equation in (\ref{EqualityDF}) we compute that
\begin{eqnarray*}
\frac{n}{n+1}\widehat{\nabla}(\widehat{\Delta}\rho)=0
\end{eqnarray*}
and then $\widehat{\Delta}\rho$ is constant on $\overline{X}$. We claim that 
$H_{\hatg}$ is a non zero constant. For this, it is enough to recall that for all $f\in C^2(\overline{X})$ it holds on $\partial X$ that
\begin{eqnarray*}
\widehat{\Delta} f =\widehat{\LSs}f+\widehat{\nabla}^2f(\widehat{N},\widehat{N}) +nH_{\hatg}\widehat{N}(f)
\end{eqnarray*}
where $\widehat{\LSs}$ denotes the Laplacian on $(\partial X,\hatgf)$. With the help of (\ref{EqualityDF}), the previous formula for $f=\rho$ yields
\begin{eqnarray*}
\widehat{\Delta}\rho =\frac{\widehat{\Delta}\rho}{n+1}-nH_{\hatg}
\end{eqnarray*}
that is 
\begin{eqnarray}\label{DeltaH}
\widehat{\Delta}\rho =-(n+1)H_{\hatg} 
\end{eqnarray}
at all points of $\partial X$. Since $\widehat{\Delta}\rho$ is a constant, the mean curvature has to be constant. If this constant is zero, the defining function is a harmonic function which vanishes on the boundary and so it has to vanish on the whole of $\overline{X}$. This gives a contradiction and then $H_{\hatg}$ is a non zero constant as well as $\Ss_{\hatgf}$ because of (\ref{EqMS}). Now using the Stokes formula in (\ref{DeltaH}) and since $\widehat{N}(\rho)=-1$ on $\partial X$, we deduce that
\begin{eqnarray*}
H_{\hatg}=\frac{1}{n+1}\frac{{\rm Vol}(\partial X,\hatgf)}{{\rm Vol}(\overline{X},\hatg)}.
\end{eqnarray*}
Since $(\overline{X},\hatg)$ is also Ricci-flat, it follows from \cite[Lemma 1]{Xia} that this compactification is isometric to a Euclidean ball. The Poincar\'e-Einstein manifold $(X,g_+)$ has to be the hyperbolic space. Conversely it is obvious to check that the unit Euclidean ball $(\mathbb{B}^{n+1},g_{\mathbb{B}})$ is a compactification of the hyperbolic space $(\mathbb{H}^{n+1},g_{\mathbb{H}})$ for which inequality (\ref{MainIneBis}) is an equality.
\end{proof}

\begin{remark}
It is easy to see from (\ref{MainCSC}) that Theorem \ref{MainRes} also holds if we only assume that the compactifications of the Poincar\'e-Einstein space have non positive constant scalar curvature.
\end{remark}

\begin{remark}
We can also note that the only compactifications of the hyperbolic space for which (\ref{MainIneBis}) is an equality are the metrics homothetic to the Euclidean metric on the unit ball. Indeed, it follows from the proof of the equality case of Theorem \ref{MainRes} that such a compactified metric has zero scalar curvature and positive constant mean curvature. From \cite[Theorem 2.1]{Escobar0}, we conclude that this metric is a constant multiple of the flat metric on the unit ball.  
\end{remark}



As a direct consequence, we obtain the following rigidity result for the hyperbolic space which generalize a recent result of \cite[Theorem 3.2]{ChenWangZhang}. It is interesting to note that no restriction is made on the conformal infinity.  
\begin{corollary}\label{RicciFlat-PE}
The only Poincar\'e-Einstein manifold which has a smooth representative in its conformal infinity and which admits a Ricci-flat conformal compactification is the hyperbolic space. 
\end{corollary}

\begin{proof}
Note that if $\hatg$ is a Ricci-flat compactification it is scalar flat so that Inequality (\ref{MainIneBis}) holds. On the other hand, the Gauss formula (\ref{GaussFormula}) implies that this inequality is in fact an equality and then the rigidity part of Theorem \ref{MainRes} allows to conclude. 
\end{proof}

It is important to remark that if we allow the Poincar\'e-Einstein manifold $(X,g_+)$ to have singularities such as cusps one can easily produce examples with Ricci-flat compactifications (also with singularities) which are not isometric to the hyperbolic space. Indeed, it is well known that the {\it hyperbolic cusps} defined by the warped products  
\begin{eqnarray*}
(X={\mathbb R}\times P,g_+=ds^2+e^{2s}\ovgf)
\end{eqnarray*}
where $(P,\ovgf)$ is a $n$-dimensional compact non-flat Ricci-flat Riemannian manifold are Einstein manifolds with scalar curvature equals to $-n(n+1)$. Then defining a new variable $t\in{\mathbb R}^+$ as $t=e^s$, we obtain
\begin{eqnarray*}
{\overline g}=e^{-2s}g_+=dt^2+\ovgf
\end{eqnarray*}
so that the Ricci-flat metric $dt^2+\ovgf$ extends to $[0,+\infty[\times P$. We conclude that a hyperbolic cusp is a Poincar\'e-Einstein manifold with $(P,[\ovgf])$ as conformal infinity and a cusp singularity at $s=+\infty$ which has a Ricci-flat compactification with one cylindrical end and which is not isometric to the hyperbolic space.



\section{Sharp lower bound for the first eigenvalue of the conformal half-Laplacian}\label{HalfConfLaplacian}


In this section, we apply Theorem \ref{MainRes} to prove a sharp inequality relating the squared of the first eigenvalue of the conformal half-Laplacian and the first eigenvalue of the conformal Laplacian of the boundary at infinity of a Poincar\'e-Einstein manifold.  

Let first recall the relation between the conformal half-Laplacian and the generalized Dirichlet-to-Neumann operator which was studied by Escobar \cite{Escobar2}. This correspondence was established by Guillarmou and Guillop\'e \cite{GuillarmouGuillope} in this situation and largely generalized in Chang and Gonzalez\cite{ChangGonzalez}. 

From the work of Graham and Zworski \cite{GrahamZworski}, it is well-known that for $f$ a smooth function on $\partial X$ and $s\in\mathbb{C}$, there exists a solution to the problem
\begin{eqnarray}\label{EigenScat}
-\Delta_{g_+} u-s(n-s)u= 0 \text{ in }X
\end{eqnarray}
such that
\begin{eqnarray}\label{ExpansionSol}
u=Fr^{n-s}+Gr^s,\quad F,G\in C^{\infty}(X),\quad F_{|r=0}=f,
\end{eqnarray}
for all $s\in\mathbb{C}$ with 
\begin{eqnarray*}
{\rm Re\,}(s)\geq\frac{n}{2},\quad s(n-s)\notin\sigma_p(-\Delta_{g_+}),\quad s\not\in\frac{n}{2}+\mathbb{N}.
\end{eqnarray*} 
Here $\sigma_p(-\Delta_{g_+})$ denotes the pure point spectrum of the Laplace operator $-\Delta_{g_+}$ on $(X,g_+)$ and $r$ the geodesic defining function corresponding to a boundary metric $\hatgf\in[\ovgf]$. The scattering operator acts on $C^\infty(\partial X)$ by
\begin{eqnarray*}
S(s)f=G_{|r=0}
\end{eqnarray*}
and defines a meromorphic family of pseudo-differential operators for ${\rm Re\,}(s)>n/2$ with simple poles of finite rank at $s=n/2+k$, $k\in\mathbb{N}$. The operators 
\begin{eqnarray}\label{GammaPower}
P^{\alpha}_{\hatgf}:=d_\alpha S\Big(\frac{n}{2}+\alpha\Big),
\end{eqnarray}
where the constant $d_\alpha$ is given in (\ref{FractionalLaplacianGZ}), are conformally covariant self-adjoint operators of order $2\alpha$ with principal symbols $(-\Delta_{\hatgf})^\alpha$. They are usually referred to as fractional conformal Laplacians. When $\alpha\in\mathbb{N}$ is an integer, it can be shown that 
\begin{eqnarray*}
{\rm Res\,}_{s=n/2+\alpha} S(s)=c_\alpha P^\alpha_{\hatgf},\quad c_\alpha=(-1)^\alpha\big(2^{2\alpha}\alpha!(\alpha-1)!\big)^{-1}
\end{eqnarray*}
are the GJMS operators constructed in \cite{GrahamJenneMasonSparling}. 

For $\alpha=1$ we recover the classical conformal Laplacian given by
\begin{eqnarray*}
P_{\hatgf}:=P_{\hatgf}^1=-\widehat{\LSs} +\frac{n-2}{4(n-1)}\Ss_{\hatgf}. 
\end{eqnarray*}
This operator appears in the famous Yamabe problem which consists to prove that, on a closed Riemannian manifold the metric can be conformally deformed to a metric with constant scalar curvature. This problem has a long and fruitful history and we refer the interested reader to \cite{LeeParker}. Its resolution relies on a deep study of the Yamabe invariant, which in our situation, is defined by
\begin{eqnarray}\label{YamabeBoundary}
\mathcal{Y}(\partial X,[\ovgf]):=\inf_{\hatgf\in[\ovgf]}\frac{\int_{\partial X} \Ss_{\hatgf}ds_{\hatgf}}{{\rm Vol}(\partial X,\hatgf)^{\frac{n-2}{n}}}.
\end{eqnarray}
This number is a conformal invariant of $(\partial X,\ovgf)$ and so is its sign. As recalled in the introduction, we remark that the non-negativity of the Yamabe invariant has deep consequences on the topology and the Riemannian structures of the associated Poincar\'e-Einstein manifold. It is also important to notice that this sign is precisely given by the one of any of the first eigenvalue of $P_{\hatgf}$ for $\hatgf\in [\ovgf]$.  Moreover, it satisfies the following inequality due to Aubin \cite{Aubin}
\begin{eqnarray}\label{YamabeSphere}
\mathcal{Y}(\partial X,[\ovgf])\leq\mathcal{Y}(\mathbb{S}^n,[\mathfrak{g}_{\mathbb{S}}])=n(n-1)\omega_n^{\frac{2}{n}}.
\end{eqnarray}
and if $\mathcal{Y}(\partial X,[\ovgf])\geq 0$ it holds that
\begin{eqnarray}\label{YamabeEigenvalues}
\mathcal{Y}(\partial X,[\ovgf])=\frac{4(n-1)}{n-2}\inf_{\hatgf\in[\ovgf]}\Big(\lambda_1(P_{\hatgf}){\rm Vol}(\partial X,\hatgf)^\frac{2}{n}\Big).
\end{eqnarray}

For $\alpha=1/2$, we obtain the so-called conformal half-Laplacian which for obvious reasons will be denoted by $\sqrt{P_{\hatgf}}$. As noticed by Guillarmou and Guillop\'e \cite{GuillarmouGuillope}, this operator is tightly related to a generalized Dirichlet-to-Neumann operator defined using the conformal Laplacian of any compactification $(\overline{X},\hatg)$ of the Poincar\'e-Einstein manifold $(X,g_+)$. Indeed first note that the existence of a solution $u$ of the eigenvalue problem (\ref{EigenScat}) for $s=(n+1)/2$ is equivalent to the fact that $U=\rho^{-\frac{n-1}{2}} u$ satisfies
\begin{eqnarray*}
-\widehat{\Delta} U+\frac{n-1}{4n}R_{\hatg}U=0,
\end{eqnarray*}
for $\hatg=\rho^2 g_+$. On the other hand, if $r$ denotes the geodesic defining function associated to $\hatgf$, we can compute that
\begin{eqnarray*}
\rho=r-H_{\hatg}r^2+O(r^3).
\end{eqnarray*}
Combining with (\ref{ExpansionSol}) yields to
\begin{eqnarray*}
U=f+\Big(S\big(\frac{n+1}{2}\big)f+\frac{n-1}{2}H_{\hatg}f\Big)r+O(r^2)
\end{eqnarray*} 
and then since $d_{1/2}=-1$ we conclude from (\ref{GammaPower}) that
\begin{eqnarray*}
\sqrt{P_{\hatg}}(f)=\big(\widehat{N}(U)+\frac{n-1}{2}H_{\hatg}U\big)_{|\partial X}
\end{eqnarray*}
where $\widehat{N}$ is the unit outward normal to $\partial X$ for $\hatg$. In particular $u$ is an eigenfunction associated to the first eigenvalue $\lambda_{1}(\sqrt{P_{\hatgf}})$ of $\sqrt{P_{\hatgf}}$ if and only if the associated $U$ satisfies 
\begin{equation}\label{StekYam}
\left\lbrace
\begin{array}{ll}
-\widehat{\Delta} U+\frac{n-1}{4n}R_{\hatg}U=0 & \quad\text{ on } \overline{X}\\ 
\widehat{N}(U)+\frac{n-1}{2}H_{\hatg}U =\lambda_{1}(\sqrt{P_{\hatgf}}) U & \quad\text{ on }\partial X.
\end{array}
\right.
\end{equation}
It is important to note here that as noticed in \cite{Escobar3} this eigenvalue can be $-\infty$. However if we assume that the boundary has non negative Yamabe invariant then $\lambda_{1}(\sqrt{P_{\hatgf}})$ is also non negative by \cite[Corollary 5]{HijaziMontiel3}. In the following, we will always assume that this assumption is fulfilled although we could only assume the finiteness of this eigenvalue for our results to hold. Anyway thanks to this characterization, we are now able to prove:
\begin{theorem}\label{YamabeSteklov}
Let $(X,g_+)$ be a Poincar\'e-Einstein manifold which has a smooth representative in its conformal infinity $(\partial X,[\ovgf])$. Then if the Yamabe invariant $\mathcal{Y}(\partial X,[\ovgf])$ is non negative, the first eigenvalue of the conformal half-Laplacian satisfies
\begin{eqnarray}\label{EigenEstimate1}
\lambda_{1}(\sqrt{P_{\hatgf}})^2\geq\frac{(n-1)^2}{n(n-2)}\lambda_{1}(P_{\hatgf})
\end{eqnarray}
for $n\geq 3$ and
\begin{eqnarray}\label{EigenEstimate2}
\lambda_{1}(\sqrt{P_{\hatgf}})^2\geq\frac{\pi}{2\,{\rm Area\,}(\partial X,\hatgf)}\chi(\partial X)
\end{eqnarray}
for $n=2$ and for any $\hatgf\in[\ovgf]$. Moreover, there exists a compactification of $(X,g_+)$ for which equality occurs in one of these inequalities if and only if $(X,g_+)$ is isometric to the hyperbolic space $(\mathbb{H}^{n+1},g_{\mathbb{H}})$.
\end{theorem}

\begin{remark}
Here $\chi(\partial X)$ denotes the Euler characteristic of the surface $\partial X$. In this situation, the non negativity of the Yamabe invariant $\mathcal{Y}(\partial X,[\ovgf])$ reduces to the topological assumption $\chi(\partial X)\geq 0$ by the Gauss-Bonnet formula. 
\end{remark}

\begin{proof}
Let $\hatg$ be a conformal compactification of $(X,g_+)$. Then consider a smooth positive eigenfunction $U\in C^\infty(\overline{X})$  associated with $\lambda_{1}(\sqrt{P_{\hatgf}})$ which therefore satisfies (\ref{StekYam}). From the relations between the scalar and the mean curvatures of two conformally related metrics (see \cite{Escobar2} for example), the metric $\widetilde{g}:=U^{\frac{4}{n-1}}\hatg$ satisfies
\begin{equation}\label{ScalZero}
R_{\widetilde{g}}=0\quad\text{and}\quad H_{\widetilde{g}}=\frac{2}{n-1}\lambda_{1}(\sqrt{P_{\hatgf}}) U^{-\frac{2}{n-1}}
\end{equation}
so that we can apply Theorem \ref{MainRes}. Then the inequality (\ref{MainIneBis}) and the second equality in (\ref{ScalZero}) give
\begin{eqnarray}\label{IneqInter}
\frac{4n}{n-1}\lambda_{1}(\sqrt{P_{\hatgf}})^2\int_{\partial X} U^{2\frac{n-2}{n-1}}\,ds_{\hatgf}\geq\int_{\partial X} \Ss_{\widetilde{\mathfrak{g}}}\,ds_{\widetilde{\mathfrak{g}}}
\end{eqnarray}
since $ds_{\widetilde{\mathfrak{g}}}=U^{\frac{2n}{n-1}}\,ds_{\hatgf}$. Inequality (\ref{EigenEstimate2}) follows from the Gauss-Bonnet formula if $n=2$. For $n\geq 3$, we write $\widetilde{\mathfrak{g}}=\widetilde{U}^\frac{4}{n-2}\hatgf$ with $\widetilde{U}=U^{\frac{n-2}{n-1}}$ and then 
\begin{eqnarray*}
\Ss_{\widetilde{\mathfrak{g}}}=\frac{4(n-1)}{n-2}\widetilde{U}^{-\frac{n+2}{n-2}}P_{\hatgf}\widetilde{U}
\end{eqnarray*}
so that Inequality (\ref{IneqInter}) finally reads
\begin{eqnarray*}
\lambda_{1}(\sqrt{P_{\hatgf}})^2\int_{\partial X} \widetilde{U}^2\,ds_{\hatgf}\geq\frac{(n-1)^2}{n(n-2)}\int_{\partial X} \widetilde{U} P_{\hatgf} \widetilde{U}\,ds_{\hatgf}.
\end{eqnarray*}
We conclude the proof with the Rayleigh characterization of the first eigenvalue $\lambda_1(P_{\hatgf})$. If equality occurs the metric $\widetilde{g}$ is such that (\ref{MainIneBis}) is an equality so that the rigidity part of Theorem \ref{MainRes} applies. Conversely, it is straightforward to see that on the round sphere $(\mathbb{S}^n,\mathfrak{g}_{\mathbb{S}})$, the non zero constant functions are eigenfunctions for $P_{\mathfrak{g}_\mathbb{S}}$ and $\sqrt{P_{\mathfrak{g}_\mathbb{S}}}$ associated respectively to $n(n-2)/4$ and to $(n-1)/2$. This implies that the unit Euclidean ball is a scalar flat compactification of the hyperbolic space for which (\ref{EigenEstimate1}) is an equality for $n\geq 3$. On the other hand, since $\chi(\mathbb{S}^2)=2$, equality also holds for $(\mathbb{B}^3,g_\mathbb{B})$ in (\ref{EigenEstimate2}) and then the same conclusion is true for $n=2$. 
\end{proof}

Now recall that the eigenvalue problem (\ref{StekYam}) is tightly related to a Yamabe-type problem on manifolds with boundary and first addressed by Escobar in \cite{Escobar2}. This problem asks for, given a smooth compact Riemannian manifold with boundary, the existence of a conformally related metric with zero scalar curvature and constant mean curvature. This is a deep and very difficult problem which is now completely solved and we refer to \cite{MayerNdiaye1} where the interested reader will be able to find a complete bibliography on this problem. As in the closed case, its resolution needs the deep study of a conformal invariant similar to the Yamabe invariant (\ref{YamabeBoundary}) which is defined by 
\begin{eqnarray}\label{EscobarInvariant}
\mathcal{E}(\overline{X},[\ovg]):=\inf_{\widehat{g}\in[\ovg]_{{\rm sf}}}\frac{\int_{\partial X}H_{\widehat{g}}\,ds_{\widehat{\mathfrak{g}}}}{{\rm Vol}(\partial X,\widehat{\mathfrak{g}})^{\frac{n-1}{n}}}
\end{eqnarray}
where 
\begin{eqnarray*}
[\ovg]_{{\rm sf}}:=\{\widehat{g}\in[\ovg]\,/\,R_{\widehat{g}}=0\}.
\end{eqnarray*}
It is useful for our purpose to recall that this invariant satisfies 
\begin{eqnarray}\label{EscobarSphere}
\mathcal{E}(\overline{X},[\ovg])\leq\mathcal{E}(\mathbb{B}^{n+1},[g_\mathbb{B}])=\omega_n^{\frac{1}{n}}
\end{eqnarray}
by Escobar \cite{Escobar2}.

As a direct application of Theorem \ref{MainRes}, we also obtain a sharp inequality which relates the conformal invariants $\mathcal{Y}(\partial X,[\ovgf])$ and $\mathcal{E}(\overline{X},[\ovg])$. This result was already proved by the author in \cite{Raulot10} and by \cite{ChenLaiWang} under additional assumptions. 
\begin{corollary}\label{main}
Let $(X,g_+)$ be a Poincar\'e-Einstein manifold with a smooth representative in its conformal infinity and such that $\mathcal{Y}(\partial X,[\ovgf])\geq 0$. Then if the dimension $n\geq 3$, we have 
\begin{eqnarray}\label{Yam1}
n(n-1)\mathcal{E}(\overline{X},[\ovg])^2\geq\mathcal{Y}(\partial X,[\ovgf])
\end{eqnarray}
and if $n=2$, 
\begin{eqnarray}\label{Yam2}
\mathcal{E}(\overline{X},[\ovg])^2\geq 2\pi\chi(\partial X).
\end{eqnarray}
Equality occurs if and only if the boundary at infinity is the M\"obius sphere $(\mathbb{S}^n,[\mathfrak{g}_{\mathbb{S}}])$. In any case, $(X,g_+)$ is isometric to the hyperbolic space $(\mathbb{H}^{n+1},g_{\mathbb{H}})$. 
\end{corollary}
\begin{proof}
From the previous discussion, there exists a metric $\hatg\in[\ovg]$ such that
\begin{eqnarray*}
R_{\hatg}=0,\quad H_{\hatg}=\mathcal{E}(\overline{X},[\ovg])\quad\text{and}\quad{\rm Vol}(\partial X,\hatgf)=1
\end{eqnarray*} 
for which Inequality (\ref{MainIneBis}) gives immediately (\ref{Yam1}) and (\ref{Yam2}). The result follows directly from the expression (\ref{YamabeBoundary}) of the Yamabe invariant $\mathcal{Y}(\partial X,[\ovgf])$. If equality occurs, the equality is also achieved in (\ref{MainIneBis}) for the metric $\hatg$ and thus the rigidity part of Theorem \ref{MainRes} allows to conclude. Conversely if the boundary at infinity is the M\"obius sphere $(\mathbb{S}^n,[\mathfrak{g}_\mathbb{S}])$ it follows from (\ref{YamabeSphere}) and (\ref{EscobarSphere}) that equality holds in (\ref{Yam1}) and (\ref{Yam2}) and so equality is also achieved in (\ref{MainIneBis}) for a compactified metric on $\overline{X}$. We conclude that $(X,g_+)$ has to be the hyperbolic space by Theorem \ref{MainRes}. 
\end{proof}

Note that this corollary contains the well-known uniqueness of the hyperbolic space as the unique Poincar\'e-Einstein manifold whose boundary at infinity is the M\"obius sphere.


\section{Upper bounds for the renormalized volume in dimension $4$}\label{UpperBoundsRenorVol}


In this section, we will see that our main integral inequality (\ref{MainIneBis}) allows to obtain a new upper bound for the renormalized volume in term of the Yamabe-type invariant $\mathcal{E}(\overline{X},[\ovg])$. 

The volume of a Poincar\'e-Einstein manifold being infinite, it has been suggested by several works that an appropriate normalization should lead to the definition of an adapted notion of volume. This has been achieved by Henningson and Skenderis \cite{HenningsonSkenderis} and Graham \cite{Graham1} as follows. Here we only consider the four dimensional case. It appears that if $r$ denotes the geodesic defining function relative to $\hatgf\in[\ovgf]$, then for $\varepsilon >0$ sufficiently small, we can compute that
\begin{eqnarray*}
{\rm Vol}_{g_+}(\{r>\varepsilon\})=c_0\varepsilon^{-3}+c_2\varepsilon^{-1}+V(X,g_+)+o(1).
\end{eqnarray*}
The coefficients $c_j$, $j=1,2$, are integrals over $\partial X$ of local curvature expressions of the metric $\hatgf$ and the constant term $V(X,g_+)$ is called the renormalized volume of the Poincar\'e-Einstein manifold $(X,g_+)$ and does not depend of the choice of $\hatgf$. An important observation by Anderson \cite{Anderson1} is that this renormalized volume appears in the Chern-Gauss-Bonnet formula in the four dimensional case. This is in fact true in the $n$-dimensional case with $n$ odd by the work of Chang, Qing and Yang \cite{ChangQingYang2}. Let us sketch the proof of their result in the four dimensional case. First recall that the Chern-Gauss-Bonnet formula on a four dimensional compact Riemannian manifold with boundary $(\overline{X},\overline{g})$ can be written for $\hatg\in[\ovg]$ as
\begin{eqnarray*}
8\pi^2\chi(\overline{X})=\frac{1}{4}\int_{\overline{X}}|W_{\hatg}|^2\,dv_{\hatg}+\int_{\overline{X}}Q_{\hatg}\,dv_{\hatg}+2\int_{\partial X}\big(\mathcal{L}_{\hatg}+T_{\hatg}\big)ds_{\hatgf}
\end{eqnarray*}
where $W_{\hatg}$ is the Weyl tensor and $Q_{\hatg}$ and $T_{\hatg}$ are respectively the $Q$-curvature and the $T$-curvature of the metric $\hatg$ which are defined by
\begin{eqnarray*}
Q_{\hatg} & := & \frac{1}{6}\big(R^2_{\hatg}-3|Ric_{\hatg}|^2-\widehat{\Delta} R_{\hatg}\big)\\
T_{\hatg} & := & \frac{1}{12}\widehat{N}(R_{\hatg})+\frac{1}{2}R_{\hatg}H_{\hatg}-\<F_{\hatg},L_{\hatg}\>+3 H_{\hatg}^3-\frac{1}{3}{\rm Tr}_{\hatg}(L^3_{\hatg})-\widehat{\LSs} H_{\hatg}.
\end{eqnarray*}
Here $F_{\hatg}$ is the covariant symmetric tensor field of order two on $\partial X$ defined by
\begin{eqnarray*}
F_{\hatg}(Y,Z)={\rm Riem}_{\hatg}(X,\widehat{N},Y,\widehat{N})
\end{eqnarray*}
for all $Y$, $Z$ tangent vector fields to $\partial X$ and where ${\rm Riem}_{\hatg}$ is the Riemann tensor of $(\overline{X},\hatg)$. Moreover the function $\mathcal{L}_{\hatg}$ is a pointwise conformal invariant which vanishes when the boundary is totally umbilical. Since this property is fulfilled for any compactification of a Poincar\'e-Einstein manifold, we can assume without loss in generality that $\mathcal{L}_{\hatg}=0$. For the same reason, the $T$-curvature reduces to 
\begin{eqnarray}\label{TcurvOmb}
T_{\hatg}=\frac{1}{12}\widehat{N}(R_{\hatg})+\frac{1}{2}H_{\hatg}\Big(R_{\hatg}-2 Ric_{\hatg}(\widehat{N},\widehat{N})\Big)+2 H^3_{\hatg}-\widehat{\LSs} H_{\hatg},
\end{eqnarray}
and the Chern-Gauss-Bonnet formula then reads
\begin{eqnarray}\label{CGBFormula}
8\pi^2\chi(\overline{X})=\frac{1}{4}\int_{\overline{X}}|W_{\hatg}|^2\,dv_{\hatg}+Q(\overline{X},\hatg),
\end{eqnarray}
where
\begin{eqnarray*}
Q(\overline{X},\hatg):=\int_{\overline{X}}Q_{\hatg}\,dv_{\hatg}+2\int_{\partial X}T_{\hatg}ds_{\hatgf}
\end{eqnarray*}
is the total $Q$-curvature of $(\overline{X},\hatg)$. Note that from the topological invariance of the Euler characteristic and the conformal invariance of the Weyl tensor, we immediately deduce the conformal covariance of the total $Q$-curvature that's why it will be denoted by $Q(\overline{X},[\ovg])$ from now.

On the other hand, Feffermann and Graham \cite{FeffermanGraham2} proved that there exists a unique function $v\in C^\infty(X)$ such that
\begin{eqnarray*}
-\Delta_{g_+} v=3
\end{eqnarray*}
and with the asymptotic
\begin{eqnarray*}
v=\log r+A+Br^3
\end{eqnarray*}
where $A$, $B\in C^\infty(X)$, even in $r$ such that $A_{|\partial X}=0$. Here $r$ denotes the defining function associated to the metric $\hatgf$ in the conformal infinity. Then they related the value of $B$ on the boundary with the renormalized volume by the very nice formula
\begin{eqnarray}\label{FGFormula}
V(X,g_+)=\int_{\partial X} B_{|\partial X}\,ds_{\hatgf}.
\end{eqnarray}
In addition to this, it can be computed that for the compactified metric $\widetilde{g}=e^{2v}g_+$ it holds that
\begin{eqnarray*}
Q_{\widetilde{g}}=0\quad\text{and}\quad T_{\widetilde{g}}=3B_{|\partial X}
\end{eqnarray*}
so that the formula (\ref{CGBFormula}) for $\widetilde{g}$ and (\ref{FGFormula}) give 
\begin{eqnarray}\label{CGB-Einstein}
8\pi^2\chi(\overline{X})=\frac{1}{4}\int_{X}|W_{\widetilde{g}}|^2\,dv_{\widetilde{g}}+6V(X,g_+).
\end{eqnarray} 
Moreover, since the $L^2$-norm of the Weyl tensor is conformally invariant, we obtain the following formula by comparing (\ref{CGBFormula}) and (\ref{CGB-Einstein}): 
\begin{proposition}\label{RV-QT}
Let $(X,g_+)$ be a four dimensional Poincar\'e-Einstein manifold then we have
\begin{eqnarray}\label{RenorVol}
V(X,g_+)=\frac{1}{6}Q(\overline{X},[\ovg]).
\end{eqnarray}
\end{proposition}
Note that this formula was already known for totally geodesic compactifications in relation with the $\sigma_2$-scalar curvature (see \cite{ChangQingYang2}). Here we just observed that it holds for {\it any} compactifications of a Poincar\'e-Einstein manifold. It was also noticed in \cite{Anderson1} that when this formula is compared to (\ref{CGBFormula}), it implies the following upper bound for the renormalized volume
\begin{eqnarray*}
V(X,g_+)\leq \frac{4\pi^2}{3}\chi(\overline{X})
\end{eqnarray*}
with equality if and only if $(X,g_+)$ is hyperbolic. Note that for the four dimensional hyperbolic space, it is not difficult to compute that
\begin{eqnarray}\label{RenVolHyperbolic}
V(\mathbb{H},g_{\mathbb{H}})=\frac{4\pi^2}{3}.
\end{eqnarray}

In this section, we first prove an upper bound for the renormalized volume in term of the Yamabe-type invariant introduced by Escobar in \cite{Escobar1}. This result is probably well-know but I didn't find it explicitly in the literature that's why a proof is included here. The Yamabe-type invariant we consider is defined by
\begin{eqnarray*}
\widetilde{\mathcal{E}}(\overline{X},[\ovg]):=\inf_{\widehat{g}\in[\ovg]_{\rm mf}}\frac{\int_{\overline{X}} R_{\widehat{g}}\,dv_{\widehat{g}}}{{\rm Vol}(\overline{X},\widehat{g})^{\frac{n-1}{n+1}}}
\end{eqnarray*}
where
\begin{eqnarray*}
[\ovg]_{{\rm mf}}:=\{\widehat{g}\in[\ovg]\,/\,H_{\widehat{g}}=0\}.
\end{eqnarray*}
It is naturally associated to the problem of proving, given a smooth Riemannian manifold with boundary, the existence of a conformally related metric with constant scalar curvature and zero mean curvature. This question was first addressed in \cite{Escobar1} and has recently been settled in \cite{MayerNdiaye2} (in which the interested reader will be able to find a complete bibliography on the subject). In our context, that is the four dimensional case with totally umbilical boundary it has been completely solved by Escobar in \cite[Theorem 4.1]{Escobar1}. We then prove:
\begin{theorem}\label{1-ReVol}
Let $(X,g_+)$ be a four dimensional Poincar\'e-Einstein manifold which has a smooth representative in its conformal infinity. Then the renormalized volume satisfies
\begin{eqnarray}\label{ReVol-1}
V(X,g_+)\leq\frac{1}{144}\widetilde{\mathcal{E}}(\overline{X},[\ovg])^2
\end{eqnarray}
and equality occurs if and only if $(X,g_+)$ is isometric to the hyperbolic space $(\mathbb{H}^4,g_{\mathbb{H}})$.
\end{theorem} 

\begin{proof}
From the formula (\ref{RenorVol}) in Proposition \ref{RV-QT}, we have
\begin{eqnarray*}
V(X,g_+)=\frac{1}{6}\int_{\overline{X}}\Big(\frac{1}{24}R_{\widehat{g}}^2-\frac{1}{2}|E_{\widehat{g}}|^2\Big)\,dv_{\widehat{g}}
\end{eqnarray*}
since $L_{\widehat{g}}=0$ for any $\widehat{g}\in[\ovg]_{{\rm mf}}$. In particular for such metrics it holds that
\begin{eqnarray}\label{bound1}
V(X,g_+)\leq\frac{1}{144}\int_{\overline{X}}R_{\widehat{g}}^2\,dv_{\widehat{g}}
\end{eqnarray} 
with equality if and only if the metric $\widehat{g}$ is Einstein. Now as explained in the beginning of this section, it follows from the work of Escobar that one can find $\widehat{g}\in[\ovg]$ such that 
\begin{eqnarray*}
R_{\widehat{g}}=\widetilde{\mathcal{E}}(\overline{X},[\ovg]),\quad H_{\widehat{g}}=0\quad\text{and}\quad{\rm Vol}(\overline{X},\widehat{g})=1.
\end{eqnarray*}
Since the boundary of any compactification of a Poincar\'e-Einstein manifold is totally umbilical, it is totally geodesic for $\widehat{g}$ since $\widehat{g}\in[\ovg]_{{\rm mf}}$. Then we can apply (\ref{bound1}) to $\widehat{g}$ and (\ref{ReVol-1}) follows directly. If equality holds, the manifold $(\overline{X},\widehat{g})$ is a compactification of a Poincar\'e-Einstein manifold which is Einstein with totally geodesic boundary, it follows form \cite[Theorem 3.1]{ChenWangZhang} that $(X,g_+)$ is the hyperbolic space. On the other hand, since from \cite{Escobar1}, it holds that
\begin{eqnarray*}
\widetilde{\mathcal{E}}(\mathbb{S}^{4}_+,[g_{\mathbb{S}}])=8\sqrt{3}\pi,
\end{eqnarray*}
it is obvious from (\ref{RenVolHyperbolic}) that the round hemisphere is a compactification of the hyperbolic space for which the inequality (\ref{ReVol-1}) is an equality.
\end{proof}

\begin{remark}
It is also a well-known result of Escobar \cite{Escobar1} that 
\begin{eqnarray*}
\widetilde{\mathcal{E}}(\overline{X},[\ovg])\leq \widetilde{\mathcal{E}}(\mathbb{S}^{4}_+,[g_{\mathbb{S}}])=8\sqrt{3}\pi
\end{eqnarray*}
with equality if and only if $(\overline{X},[\ovg])$ is the four dimensional round hemisphere $(\mathbb{S}^{4}_+,[g_{\mathbb{S}}])$. In particular the inequality (\ref{ReVol-1}) gives an alternative proof of the well-known inequality \cite{ChangQingYang1}
\begin{eqnarray}\label{RV-Upper}
V(X,g_+)\leq V(\mathbb{H},g_\mathbb{H})=\frac{4\pi^2}{3},
\end{eqnarray}
with equality if and only if $(X,g_+)$ is isometric to the hyperbolic space. 
\end{remark}

As a direct consequence, we easily deduce:
\begin{corollary}
Let $(X,g_+)$ be a four dimensional Poincar\'e-Einstein manifold which has a smooth representative in its conformal infinity. If $\widetilde{\mathcal{E}}(\overline{X},[\ovg])=0$ then the renormalized volume of $(X,g_+)$ satisfies $V(X,g_+)<0$.
\end{corollary}

It is then a natural question to ask if one can relate the renormalized volume with the Yamabe-type invariant $\mathcal{E}(\overline{X},[\ovg])$ defined by (\ref{EscobarInvariant}). We will see that this is indeed possible but the proof is much more involved since it relies on the integral inequality (\ref{MainIneBis}). First note that the Gauss formula (\ref{GaussFormula}) in the expression (\ref{TcurvOmb}) gives  
\begin{eqnarray*}
T_{\hatg}=\frac{1}{12}\widehat{N}(R_{\hatg})+\frac{1}{2}H_{\hatg}\Big(\Ss_{\hatgf}-6H^2_{\hatg}\Big)+2 H^3_{\hatg}-\widehat{\LSs} H_{\hatg}.
\end{eqnarray*}
Then if $\widehat{g}\in[\ovg]_{{\rm sf}}$, Proposition \ref{RV-QT} implies that
\begin{eqnarray*}
V(X,g_+)=-\frac{1}{12}\int_{\overline{X}}|Ric_{\widehat{g}}|^2\,dv_{\widehat{g}}+\frac{1}{6}\int_{\partial X}H_{\widehat{g}}\Big(\Ss_{\widehat{\mathfrak{g}}}-6H^2_{\widehat{g}}\Big)\,ds_{\widehat{\mathfrak{g}}}+\frac{2}{3}\int_{\partial X}H^3_{\widehat{g}}\,ds_{\widehat{\mathfrak{g}}}
\end{eqnarray*} 
that is
\begin{eqnarray}\label{QcurvatureOmb1}
\frac{2}{3}\int_{\partial X}H^3_{\widehat{g}}\,ds_{\widehat{\mathfrak{g}}}\geq V(X,g_+)+\frac{1}{6}\int_{\partial X}H_{\widehat{g}}\Big(6H^2_{\widehat{g}}-\Ss_{\widehat{\mathfrak{g}}}\Big)\,ds_{\widehat{\mathfrak{g}}},
\end{eqnarray}
with equality if and only if $\hatg$ is a Ricci-flat metric. Now it follows from \cite[Theorem 6.1]{Escobar2} that there exists $\widehat{g}\in[\ovg]$ such that
\begin{eqnarray*}
R_{\widehat{g}}=0,\quad H_{\widehat{g}}=\mathcal{E}(\overline{X},[\ovg])\quad\text{and}\quad{\rm Vol}(\partial X,\widehat{\mathfrak{g}})=1,
\end{eqnarray*}
and for which the inequality (\ref{QcurvatureOmb1}) reads as
\begin{eqnarray*}
\frac{2}{3}\mathcal{E}(\overline{X},[\ovg])^3\geq V(X,g_+)+\frac{1}{6}\mathcal{E}(\overline{X},[\ovg])\int_{\partial X}\Big(6H^2_{\widehat{g}}-\Ss_{\widehat{\mathfrak{g}}}\Big)\,ds_{\widehat{\mathfrak{g}}}.
\end{eqnarray*}
Since $\widehat{g}\in[\ovg]_{{\rm sf}}$, the inequality (\ref{MainIneBis}) applies so that the second term in the right-hand side of the previous inequality is non-negative as soon as $\mathcal{E}(\overline{X},[\ovg])\geq 0$. We then finally obtain:
\begin{theorem}\label{RV4}
Let $(X,g_+)$ be a four dimensional Poincar\'e-Einstein manifold which has a smooth representative in its conformal infinity. Then if $\mathcal{Y}(\partial X,[\ovgf])\geq 0$ it holds that 
\begin{eqnarray}\label{ReVol-2}
 V(X,g_+)\leq\frac{2}{3}\mathcal{E}(\overline{X},[\ovg])^3.
\end{eqnarray}
Moreover equality occurs if and only if $(X,g_+)$ is isometric to the hyperbolic space $(\mathbb{H}^4,g_{\mathbb{H}})$.
\end{theorem}

\begin{proof}
The previous discussion holds since, as explained in Section \ref{HalfConfLaplacian}, the non negativity of $\mathcal{Y}(\partial X,[\ovgf])$ implies the non negativity of $\mathcal{E}(\overline{X},[\ovg])$. On the other hand, if equality holds in (\ref{ReVol-2}) we conclude that $(\overline{X},\widehat{g})$ is a Ricci-flat conformal compactification of the Poincar\'e-Einstein manifold $(X,g_+)$ and so Corollary \ref{RicciFlat-PE} implies that it has to be the hyperbolic space. Finally, since from (\ref{EscobarSphere}) we have 
\begin{eqnarray*}
\mathcal{E}(\mathbb{B}^{4}_+,[g_{\mathbb{B}}])^3=2\pi^2,
\end{eqnarray*}
it follows from (\ref{RenVolHyperbolic}) that the unit Euclidean ball is a compactification of the hyperbolic space for which the inequality (\ref{ReVol-2}) is an equality.
\end{proof}

Note that combining (\ref{EscobarSphere}) with (\ref{ReVol-2}) gives an alternative proof of the inequality (\ref{RV-Upper}).


\section{Lower bounds for the first eigenvalue of Dirac operators}\label{DiracEstimates}


In this section, we prove new lower bounds for the first non negative eigenvalue of Dirac operators in the context of Poincar\'e-Einstein manifolds. 

Let us briefly recall some standard facts on spin manifolds. For more details on this subject, we refer to \cite{BourguignonHijaziMilhoratMoroianu,Friedrich,Ginoux,LawsonMichelsohn}. On a $(n+1)$-dimensional compact Riemannian spin manifold $(\overline{X},\hatg)$ with boundary, there exists a smooth Hermitian vector bundle over $\overline{X}$ called the spinor bundle which will be denoted by $\Sigma\overline{X}$. The sections of this bundle are called spinors. Moreover, the tangent bundle $T\overline{X}$ acts on $\Sigma\overline{X}$ by Clifford multiplication $Y\otimes \psi\mapsto \widehat{\gamma}(Y)\psi$ for any tangent vector fields $Y$ and any spinor fields $\psi$. On the other hand, the Riemannian Levi-Civita connection $\widehat{\nabla}$ lifts to the so-called spin Levi-Civita connection (also denoted by $\widehat{\nabla}$) and defines a metric connection on $\Sigma\overline{X}$ that preserves the Clifford multiplication. The Dirac operator is then the first order elliptic differential operator acting on the spinor bundle $\Sigma\overline{X}$ given by $D_{\hatg}:=\widehat{\gamma}\circ\widehat{\nabla}$. Moreover, the spin structure on $\overline{X}$ induces (via the unit normal field to $\partial X$) a spin structure on its boundary which allows to define the {\it extrinsic} spinor bundle $\RSB:=\Sigma\overline{X}_{|\partial X}$ over $\partial X$ on which there exists a Clifford multiplication $\widehat{\mult}$ and a metric connection $\widehat{\nbs}$. Similarly, the extrinsic Dirac operator is defined by taking the Clifford trace of the covariant derivative $\widehat{\nbs}$ that is $\D_{\hatgf}:=\widehat{\mult}\circ\widehat{\nbs}$. From the spin structure on $\partial X$, one can also construct an {\it intrinsic} spinor bundle for the induced metric $\hatgf$ denoted by $\Sigma\partial X$ and endowed with a Clifford multiplication $\widehat{\gamma}^{\partial X}$ and a spin Levi-Civita connection $\widehat{\nabla}^{\partial X}$. Note that the (intrinsic) Dirac operator on $(\partial X,\hatgf)$ is obviously defined by $\widehat{\D}=\widehat{\gamma}^{\partial X}\circ\widehat{\nabla}^{\partial X}$. In fact, we have an isomorphism 
$$
\big(\RSB,\widehat{\nbs},\widehat{\mult}\big)\simeq
\left\lbrace
\begin{array}{ll}
\big(\Sigma\partial X,\widehat{\nabla}^{\partial X},\widehat{\gamma}^{\partial X}\big) & \text{ if } n \text{ is even}\\
\big(\Sigma\partial X,\widehat{\nabla}^{\partial X},\widehat{\gamma}^{\partial X}\big)\oplus\big(\Sigma\partial X,\widehat{\nabla}^{\partial X},-\widehat{\gamma}^{\partial X}\big)& \text{ if } n \text{ is odd}
\end{array}
\right.
$$
so that the restriction of a spinor field on $\overline{X}$ to $\partial X$ and the extension of a spinor field on $\partial X$ to $\overline{X}$ are well-defined. These identifications also imply in particular that the spectrum of the extrinsic Dirac operator is an intrinsic invariant of the boundary which means that it only depends on the spin and Riemannian structures of $\partial X$ and not on how it is embedded in $\overline{X}$.  

As noticed in \cite{AnderssonDahl,HijaziMontiel3}, it is important to pay attention to the involved spin structures. Since a Poincar\'e-Einstein manifold is topologically given by the interior of a compact and connected manifold with boundary, we will say that $X$ is spin if the manifold with boundary $\overline{X}=X\cup\partial X$ is spin and then $\partial X$ will always be endowed with the induced spin structure.


\subsection{The boundary Dirac operator}\label{BoundaryDiracOperator}


In this part, we apply our results to the first non negative eigenvalue of the boundary Dirac operator $\D_{\hatgf}$. In view of the previous discussion, all the next results also hold for the absolute value of the first non zero eigenvalue of the intrinsic Dirac operator. For this, we need an inequality proved by Hijazi, Montiel and Zhang in \cite{HijaziMontielZhang2} in the context of compact Riemannian spin manifolds with boundary which relates this eigenvalue, denoted by $\lambda_1(\D_{\hatgf})$, and the first eigenvalue of the problem (\ref{StekYam}). In our setting, we can refine the equality case and state their result as follow:
\begin{theoremsub}\label{HMZ}
Let $(X,g_+)$ be a spin Poincar\'e-Einstein manifold which has a smooth representative in its conformal infinity $(\partial X,[\ovgf])$. Then if the Yamabe invariant $\mathcal{Y}(\partial X,[\ovgf])$ is non negative, the first eigenvalue $\lambda_1(\D_{\widehat{\mathfrak{g}}})$ of the boundary Dirac operator satisfies
\begin{eqnarray}\label{HijaziMontielZhang}
\lambda_1(\D_{\widehat{\mathfrak{g}}})\geq\frac{n}{n-1}\lambda_{1}(\sqrt{P_{\widehat{\mathfrak{g}}}})
\end{eqnarray}  
for any $\hatg\in[\ovg]$ with $\ovg=\rho^2 g_+$ and $\rho$ is a smooth defining function. Moreover, there exists a compactification of $(X,g_+)$ for which (\ref{HijaziMontielZhang}) is an equality if and only if $(X,g_+)$ is isometric to the hyperbolic space $(\mathbb{H}^{n+1},g_{\mathbb{H}})$.  
\end{theoremsub}

\begin{proof}
The inequality is exactly the one obtained in \cite[Theorem 9]{HijaziMontielZhang2} (the constant is different because of our choice of normalization). Moreover, if the equality occurs in (\ref{HijaziMontielZhang}) it is proved in the same result that the manifold $(\overline{X},\hatg)$ has a Ricci-flat metric in its conformal class. In this situation, Corollary \ref{RicciFlat-PE} applies and the conclusion follows immediately. Conversely, it is well-known that $\lambda_1(\D_{\mathfrak{g}_{\mathbb{S}}})=n/2$ on the unit round sphere as well as $\lambda_{1}(\sqrt{P_{\mathfrak{g}_{\mathbb{S}}}})=(n-1)/2$, so that equality occurs in (\ref{HijaziMontielZhang}) for the unit Euclidean ball. In this situation, the associated Poincar\'e-Einstein is the hyperbolic space.  
\end{proof}

\begin{remarksub}
In \cite{ChenWangZhang}, the authors note that the main inequality in \cite{HijaziMontielZhang2} was proved for manifolds with boundary whose boundary is an internal hypersurface. This is due to the fact that it relies on the unique continuation property which is well-known to hold in this situation. However, from \cite[Section 1.2]{BoossBavnbekLesch} this property also hold if we assume that the hypersurface is the ``true'' boundary of the manifold so that Inequality (\ref{HijaziMontielZhang}) holds in our situation (and in fact for all compact manifolds with boundary as soon as the other assumptions in \cite[Theorem 9]{HijaziMontielZhang2} are fulfilled). 
\end{remarksub}

Now if we assume that the Yamabe invariant $\mathcal{Y}(\partial X,[\ovgf])$ of the boundary at infinity is non negative, which implies that $\mathcal{E}(\overline{X},[\ovg])\geq 0$ as discussed in Section \ref{HalfConfLaplacian}, we can apply for $\widehat{g}\in[\ovg]$ the inequalities (\ref{EigenEstimate1}) and (\ref{EigenEstimate2}) to get 
\begin{eqnarray}\label{HijaziClosed}
\lambda_1(\D_{\widehat{\mathfrak{g}}})^2\geq\frac{n}{n-2}\lambda_{1}(P_{\widehat{\mathfrak{g}}})
\end{eqnarray}
for $n\geq 3$ and 
\begin{eqnarray}\label{BarClosed}
\lambda_1(\D_{\widehat{\mathfrak{g}}})^2\geq\frac{2\pi}{{\rm Area\,}(\partial X,\widehat{\mathfrak{g}})}\chi(\partial X)
\end{eqnarray}
for $n=2$. Moreover equality occurs if and only if $(\partial X,\widehat{\mathfrak{g}})$ carries a real Killing spinor. These inequalities are in fact well-known in the context of closed spin manifolds, the first one due to Hijazi \cite{Hijazi1} and known as the Hijazi inequality and the second one due to B\"ar \cite{Bar1} and Hijazi \cite{Hijazi2}. Note that Inequality (\ref{HijaziClosed}) differs from the original one from a constant since we take an other normalization in the definition of the conformal Laplacian $P_{\widehat{\mathfrak{g}}}$. Recall also that manifolds carrying real Killing spinors are Einstein and that, in the simply connected case, they have been classified by B\"ar \cite{Bar2}. They are round spheres, Einstein-Sasaki manifolds, $3$-Sasaki manifolds, nearly-K\"ahler non-K\"ahler $6$-manifolds and $7$-manifolds carrying nice $3$-forms. The novelty in our approach is that when this type of manifolds arises as the boundary of a spin conformal compactification of a Poincar\'e-Einstein manifold then the inequality (\ref{EigenEstimate1}) (resp. (\ref{EigenEstimate2})) is always sharper than the inequality (\ref{HijaziClosed}) (resp. (\ref{BarClosed})). In particular, assume that $(X,g_+)$ is a Poincar\'e-Einstein manifold whose boundary at infinity carries a real Killing spinor that is there exists a metric $\hatgf\in[\ovgf]$ for which such a spinor field exists. Then it follows that equality occurs in (\ref{HijaziClosed}) and, from Theorem \ref{YamabeSteklov}, equality also holds in (\ref{EigenEstimate1}). This implies that $(X,g_+)$ has to be the hyperbolic space and $(\partial X,[\ovgf])$ is the M\"obius sphere. We thus have proved that the only spin Poincar\'e-Einstein manifold with a real Killing spinor in its conformal infinity is the hyperbolic space. In particular, non spherical manifolds with a real Killing spinor cannot be the conformal infinity of a spin Poincar\'e-Einstein manifold. In fact, with a little more effort, we can say more. Indeed assume that the boundary at infinity carries a twistor-spinor $\varphi\in\Gamma(\Sigma\partial X)$ that is a section of the spinor bundle over $(\partial X,\hatgf)$ such that
\begin{eqnarray*}
\widehat{\nabla}^{\partial X}_Y\varphi+\frac{1}{n}\widehat{\gamma}^{\partial X}(Y)\widehat{\D}\varphi=0
\end{eqnarray*}
for all $Y\in\Gamma(T\partial X)$. Note that this property only depends on the conformal structure of the underlying manifold so that it is a natural assumption to impose on the boundary at infinity. It can be seen as a supersymmetric version of the symmetric condition which assume the existence of a conformal Killing vector on $(\partial X,[\ovgf])$. Moreover, if a twistor spinor is also an eigenspinor for the Dirac operator it is a Killing spinor. In particular, this situation contains the previous case where real Killing spinors were considered. For more details on the subject, we refer to \cite{BaumFriedrichGrunewaldKath} for example. Let us now state the result which follows directly from our reasoning. As pointed out in the introduction, this result was allready known and proved under weaker assumptions by Hijazi and Montiel \cite{HijaziMontiel3} but it seems interesting to the author to give an alternative proof.
\begin{theoremsub}\label{HM-Twistor}
A spin Poincar\'e-Einstein manifold $(X,g_+)$ whose conformal infinity $(\partial X,[\ovgf])$ has a smooth representative and which carries a twistor spinor is isometric to the hyperbolic space $(\mathbb{H}^{n+1},g_{\mathbb{H}})$. 
\end{theoremsub}  

\begin{proof}
From the solution of the Yamabe problem \cite{Schoen}, we can choose a metric $\hatgf\in[\ovgf]$ with constant scalar curvature $\Ss_{\hatgf}$ and whose sign is precisely given by the sign of its Yamabe invariant $\mathcal{Y}(\partial X,[\ovgf])$. Moreover, from the variational characterization of the first eigenvalue of the conformal Laplacian we deduce that 
\begin{eqnarray}\label{CL-Estimate}
\lambda_{1}(P_{\hatgf})=\inf_{f\in C^\infty(\partial X)}\frac{\int_{\partial X}\Big(|\widehat{\nbs} f|^2+\frac{n-2}{4(n-1)}\Ss_{\hatgf} f^2\Big)\,ds_{\hatgf}}{\int_{\partial X}f^2\,ds_{\hatgf}}\geq \frac{n-2}{4(n-1)}\Ss_{\hatgf}.
\end{eqnarray}
On the other hand, there exists on the spinor bundle over $\partial X$ endowed with the metric $\hatgf$ a twistor spinor and it follows from \cite{Lichnerowicz2} that it is an eigenspinor for the squared of the Dirac operator $\widehat{\D}$ associated to the eigenvalue $\frac{n\Ss_{\hatgf}}{4(n-1)}$. Note that since $\partial X$ is compact and the operator $\widehat{\D}^2$ is non negative, the scalar curvature with respect to the metric $\hatgf$ is a non negative constant so that $\mathcal{Y}(\partial X,[\ovgf])\geq 0$. For $n\geq 3$, we can then apply (\ref{EigenEstimate1}) in Theorem \ref{YamabeSteklov}, the Hijazi inequality (\ref{HijaziClosed}) and (\ref{CL-Estimate}) to get
\begin{eqnarray*}
\frac{n\Ss_{\hatgf}}{4(n-1)}\geq\lambda_1(\D_{\hatgf})^2\geq\frac{n^2}{(n-1)^2}\lambda_{1}(\sqrt{P_{\hatgf}})^2 \geq\frac{n}{n-2}\lambda_{1}(P_{\hatgf})\geq \frac{n\Ss_{\hatgf}}{4(n-1)}
\end{eqnarray*}
so that the equality holds in (\ref{EigenEstimate1}) and $(X,g_+)$ is isometric to the hyperbolic space as claimed. The same arguments hold for $n=2$ using (\ref{EigenEstimate2}) in Theorem \ref{YamabeSteklov}. 
\end{proof}
This result also prevents the existence of spin Poincar\'e-Einstein manifold whose boundary at infinity carries a parallel spinor field. These manifolds include Calabi-Yau manifolds, hyper-K\"ahler manifolds, $G_2$ $7$-manifolds, ${\rm Spin}_7$ $8$-manifolds in the simply connected case \cite{WangMK} as well as the flat tori equipped with the trivial spin structure.  

To conclude this section, we combine the estimate of Hijazi, Montiel and Zhang and the upper bound (\ref{ReVol-2}) to obtain a lower bound of the first non negative eigenvalue of the boundary Dirac operator in term of the renormalized volume of $(X,g_+)$. We first need to remark that from the variational characterizations of the eigenvalue $\lambda_{1}(\sqrt{P_{\hatgf}})$, the conformal invariant $\mathcal{E}(\overline{X},[\ovg])$ and the H\"older inequality, we immediately deduce \cite[Corollary 11]{HijaziMontielZhang2} that
\begin{eqnarray*}
\lambda_1(\D_{\widehat{\mathfrak{g}}}){\rm Vol}(\partial X,\widehat{\mathfrak{g}})^{\frac{1}{n}}\geq\frac{n}{2}\mathcal{E}(\overline{X},[\ovg])
\end{eqnarray*}  
for any $\hatg\in[\ovg]$. Then it is straightforward to deduce from Theorem \ref{RV4} that in the four dimensional case we get:
\begin{theoremsub}\label{ReVol-BDO1}
Let $(X,g_+)$ be a four dimensional spin Poincar\'e-Einstein manifold such that the boundary at infinity has a smooth representative and non negative Yamabe invariant. Then the first non negative eigenvalue $\lambda_1(\D_{\hatgf})$ of the boundary Dirac operator satisfies
\begin{eqnarray}\label{ReVol-BDO}
\lambda_1(\D_{\widehat{\mathfrak{g}}})^3{\rm Vol}(\partial X,\widehat{\mathfrak{g}})\geq\frac{81}{16}V(X,g_+)
\end{eqnarray}
for any $\hatg\in[\ovg]$. Moreover, there exists a compactification of $(X,g_+)$ for which (\ref{ReVol-BDO}) is an equality if and only if $(X,g_+)$ is isometric to the hyperbolic space $(\mathbb{H}^{n+1},g_{\mathbb{H}})$. 
\end{theoremsub}

It is maybe interesting to note that this bound on the renormalized volume only depends on the geometry of the boundary at infinity while the renormalized volume depends, a priori, on the global geometry of the bulk manifold $(X,g_+)$. This inequality suggests the definition of the following spin conformal invariant 
\begin{eqnarray*}
\mathcal{\D}(\overline{X},[\ovg]):=\inf_{\hatg\in[\ovg]}\Big(\lambda_1(\D_{\widehat{\mathfrak{g}}}){\rm Vol}(\partial X,\widehat{\mathfrak{g}})^\frac{1}{n}\Big),
\end{eqnarray*}
where $(\overline{X},[\overline{g}])$ is a $(n+1)$-dimensional conformal spin compact manifold with boundary. It can be seen as an extrinsic analogue of the invariant studied by Ammann \cite{Ammann1} (see also the references therein) which is tightly related to the existence of periodic constant mean curvature immersions of surfaces \cite{Ammann2}. We leave the study of this invariant for further investigations.


\subsection{The condition associated to a chirality operator}\label{BoundaryCHI}


In this section, we assume that there exists a {\it chirality operator} $\mathcal{G}$ on $\overline{X}$ that is an unitary and parallel involution of the spinor bundle $\Sigma\overline{X}$ which anticommutes with the action of a tangent vector field on $\overline{X}$. Note that such a map always exists when $n$ is odd. It is then well-known \cite{HijaziMontielRoldan1} that the projection $P_{\mathcal{G}}:=\frac{1}{2}\big({\rm Id}-\widehat{\gamma}(\widehat{N})\mathcal{G}\big)$ defines a local elliptic boundary condition for the Dirac operator for which the spectrum is a discrete unbounded sequence of real eigenvalues with finite dimensional eigenspaces. If $\lambda_1(D_{\hatg})$ denotes this first positive eigenvalue, it is proved in \cite{Raulot2} by the author that 
\begin{eqnarray*}
\lambda_1(D_{\hatg})^2{\rm Vol}(\overline{X},\hatg)^{\frac{2}{n+1}}\geq\frac{n+1}{4n}\widetilde{\mathcal{E}}(\overline{X},[\ovg])
\end{eqnarray*} 
and so we can combine this estimate with Theorem \ref{1-ReVol} to deduce:
\begin{theoremsub}\label{Chi-ReVol}
Let $(X,g_+)$ be a four dimensional spin Poincar\'e-Einstein manifold which has a smooth representative in its conformal infinity. Then the first eigenvalue $\lambda_1(D_{\hatg})$ of the Dirac operator under the condition associated to a chirality operator satisfies
\begin{eqnarray}\label{ReVol-Chi}
\lambda_1(D_{\hatg})^4{\rm Vol}(\overline{X},\hatg)^{2}\geq 16 V(X,g_+)
\end{eqnarray}
for any $\hatg\in[\ovg]$. Moreover, there exists a compactification of $(X,g_+)$ for which (\ref{ReVol-Chi}) is an equality if and only if $(X,g_+)$ is isometric to the hyperbolic space $(\mathbb{H}^{n+1},g_{\mathbb{H}})$.
\end{theoremsub}

Finally, it is interesting to remark that one can formulate all these estimates in a more conformal way, meaning that all the data involved depend only on the conformal structure of the compactification of the Poincar\'e-Einstein manifold unlike the left-hand side of Inequality (\ref{ReVol-Chi}). This formulation needs the introduction of a spin conformal invariant defined by the author in \cite{Raulot4} whose expression for four dimensional manifolds is
\begin{eqnarray*}
\mathcal{D}(\overline{X},[\overline{g}]):=\inf_{\widehat{g}\in[\overline{g}]}\Big(\lambda_1(D_{\widehat{g}}){\rm Vol}(\overline{X},\widehat{g})^\frac{1}{4}\Big).
\end{eqnarray*}
This finally leads to:
\begin{corollarysub}
Under the assumptions of Theorem \ref{Chi-ReVol}, it holds that:
\begin{eqnarray*}
16 V(X,g_+)\leq 9\widetilde{\mathcal{E}}(\overline{X},[\ovg])^2\leq\mathcal{D}(\overline{X},[\overline{g}])^4\leq \frac{64\pi^2}{3}.
\end{eqnarray*}
\end{corollarysub}
\begin{proof}
The last inequality, which is the only thing left to prove, follows from \cite[Theorem 9]{Raulot4}. 
\end{proof}


\bibliographystyle{alpha}
\bibliography{Biblio-PE} 

\end{document}